\documentclass[12pt,twoside]{amsart}
\usepackage[latin1]{inputenc}
\usepackage{amsmath, amsthm, amscd, amsfonts, amssymb, graphicx}
\usepackage[bookmarksnumbered, plainpages]{hyperref}

\newtheorem{thm}{Theorem}[section]

\newtheorem{lem}[thm]{Lemma}
\newtheorem{prop}[thm]{Proposition}
\newtheorem{defn}[thm]{Definition}

\numberwithin{equation}{section}

\allowdisplaybreaks

\begin{document}

\title{\bf Super warped products with a semi-symmetric metric connection}

\author{Yong Wang}

\thanks{{\scriptsize
\hskip -0.4 true cm \textit{2010 Mathematics Subject Classification:}
53C40; 53C42.
\newline \textit{Corresponding author:} Yong Wang}}

\maketitle

\begin{abstract}
 In this paper, we define the semi-symmetric metric
 connection on super Riemannian manifolds. We compute the semi-symmetric metric
 connection and its curvature tensor and its Ricci tensor on super warped product spaces. We introduce two kind of super warped product spaces with the semi-symmetric metric
 connection and give the conditions under which these two super warped product spaces with the semi-symmetric metric
 connection are the Einstein super spaces with the semi-symmetric metric
 connection.

\end{abstract}

\vskip 0.2 true cm


\pagestyle{myheadings}
\markboth{\rightline {\scriptsize Wang}}
         {\leftline{\scriptsize Super warped products}}

\bigskip
\bigskip


\section{ Introduction}
 \indent  \quad  The (singly) warped product $B\times_hF$ of two pseudo-Riemannian manifolds $(B,g_B)$ and
    $(F,g_F)$ with a smooth function $h:B\rightarrow (0,\infty)$ is the product
    manifold $B\times F$ with the metric tensor $g=g_B\oplus
    h^2g_F.$ Here, $(B,g_B)$ is called the base manifold and
    $(F,g_F)$ is called as the fiber manifold and $h$ is called as
    the warping function. Generalized Robertson-Walker space-times
    and standard static space-times are two well-known warped
    product spaces.The concept of warped products was first introduced by
    Bishop and ONeil (see \cite{BO}) to construct examples of Riemannian
    manifolds with negative curvature. In Riemannian geometry,
    warped product manifolds and their generic forms have been used
    to construct new examples with interesting curvature properties
    since then. In \cite{DD}, F. Dobarro and E. Dozo had studied from the viewpoint of partial differential equations and variational methods,
    the problem of showing when a Riemannian metric of constant scalar curvature can be produced on a product manifolds by a warped product
    construction.
    In \cite{EJK}, Ehrlich, Jung and Kim got explicit solutions to warping function to have a constant scalar curvature for generalized
    Robertson-Walker space-times.
    In \cite{ARS}, explicit solutions were also obtained for the warping
    function to make the space-time as Einstein when the fiber is
    also Einstein.\\
      \indent The definition of a semi-symmetric metric connection was given by H. Hayden in \cite{Ha}. In 1970, K. Yano \cite{Ya}
        considered a semi-symmetric metric connection and studied some of its properties. He proved that a Riemannian manifold admitting
         the semi-symmetric metric connection has vanishing curvature tensor if and only if it is conformally flat. Motivated by the Yano'
         result,
         in \cite{SO}, Sular and \"{O}zgur studied warped product manifolds with a
         semi-symmetric metric connection, they computed curvature of semi-symmetric metric connection
          and considered Einstein warped product manifolds with a semi-symmetric metric connection. In \cite{W1}, we extended the results of Sular and \"{O}zgur to multiply twisted products with a semi-symmetric metric connection.\\
  \indent On the other hand, in \cite{BG}, the definition of super warped product spaces was given. In \cite{GDMVR}, several new super warped product spaces were given and the authors also studied the Einstein
  equations with cosmological constant in these new super warped product spaces. Our motivation is to extend the results of Sular and \"{O}zgur to super warped product spaces with a semi-symmetric metric connection.\\
\indent In Section 2, we state some definitions of super manifolds and super Riemannian metrics. We also define the semi-symmetric metric
 connection on super Riemannian manifolds and prove that there is a unique semi-symmetric metric
 connection on super Riemannian manifolds which is metric and has the semi-symmetric torsion. In Section 3, we compute the semi-symmetric metric
 connection and its curvature tensor and its Ricci tensor on super warped product spaces. In Section 4, we introduce two kind of super warped product spaces with the semi-symmetric metric
 connection and give the conditions under which these two super warped product spaces with the semi-symmetric metric
 connection are the Einstein super spaces with the semi-symmetric metric
 connection.


\vskip 1 true cm

\section{A semi-symmetric metric
 connection on super Riemannian manifolds}

Firstly we introduce some notations on Riemannian supergeometry.
\begin{defn} A locally $\mathbb{Z}_2$-ringed space is a pair $S:= (|S|, \mathcal{O}_S)$ where $|S|$ is a second-countable
Hausdorff space, and a $\mathcal{O}_S$ is a sheaf of $\mathbb{Z}_2$-graded $\mathbb{Z}_2$-commutative associative unital $\mathbb{R}$-algebras, such that the
stalks $\mathcal{O}_{S,p}$, $p\in |S|$ are local rings.
\end{defn}
 \indent In this context, $\mathbb{Z}_2$-commutative means that any two sections $s,t\in \mathcal{O}_S(|U|),~~|U|\subset|S|$ open,
 of homogeneous
degree $|s|\in \mathbb{Z}_2$
and $|t|\in \mathbb{Z}_2$
commute up to the sign rule
$st=(-1)^{|s||t|}ts$.
 $\mathbb{Z}_2$-ring
space $U^{m|n}:= (U,C^{\infty}_{U^m}\otimes \wedge \mathbb{R}^n)$, is called standard
superdomain where $C^{\infty}_{U^m}$ is the sheaf of smooth functions on $U$ and $\wedge\mathbb{R}^n$ is
the exterior algebra of $\mathbb{R}^n$. We can employ (natural) coordinates $x^I:=(x^a,\xi^A)$ on any $\mathbb{Z}_2$-domain, where $x^a$ form a coordinate system on $U$ and the $\xi^A$
are formal coordinates.
\begin{defn}
 A supermanifold of dimension $m|n$ is a super ringed space
$M=(|M|, \mathcal{O}_M )$ that is locally isomorphic to $\mathbb{R}^{m|n}$ and $|M|$ is a second countable
and Hausdorff topological space.
\end{defn}
 The tangent sheaf $\mathcal{T}M$ of a $\mathbb{Z}_2$-manifold $M$ is defined as the sheaf of derivations of sections of the structure
sheaf, i.e., $\mathcal{T}M(|U|) := {\rm Der}(\mathcal{O}_M(|U|)),$ for arbitrary open set $|U|\subset |M|.$ Naturally, this is a sheaf of locally free $\mathcal{O}_M$-modules. Global sections of the tangent sheaf are referred to as {\it vector fields}. We denote the $\mathcal{O}_M (|M|)$-module
of vector fields as ${\rm Vect}(M)$. The dual of the tangent sheaf is the {\it cotangent sheaf}, which we denote as $\mathcal{T}^*M$.
This is also a sheaf of locally free $\mathcal{O}_M$-modules. Global section of the cotangent sheaf we will refer to as {\it one-forms}
and we denote the $\mathcal{O}_M(|M|)$-module of one-forms as $\Omega^1(M)$.
\begin{defn}
 A Riemannian metric on a $\mathbb{Z}_2$-manifold M is a $\mathbb{Z}_2$-homogeneous, $\mathbb{Z}_2$-symmetric, non-degenerate,
$\mathcal{O}_M$-linear morphisms of sheaves $\left<-,-\right>_g:~~\mathcal{T}M\otimes \mathcal{T}M\rightarrow \mathcal{O}_M.$
A $\mathbb{Z}_2$-manifold equipped with a Riemannian metric is referred to as a Riemannian $\mathbb{Z}_2$-manifold.
\end{defn}
We will insist that the Riemannian metric is homogeneous with respect to the $\mathbb{Z}_2$-degree, and we will denote
the degree of the metric as $|g| \in \mathbb{Z}_2$.
Explicitly, a Riemannian metric has the following properties:\\
(1)$ |\left<X,Y\right>_g |= |X| + |Y |+ |g|,$\\
(2)$\left<X,Y\right>_g =(-1)^{|X||Y|}\left<Y,X\right>_g,$\\
(3) If $\left<X,Y\right>_g = 0$ for all $Y \in Vect(M),$ then $X = 0,$\\
(4) $\left<fX+Y,Z\right>_g =f\left<X,Z\right>_g +\left<Y,Z\right>_g ,$\\
for arbitrary (homogeneous) $ X, Y, Z \in {\rm Vect}(M)$ and $f \in C^{\infty}(M)$. We will say that a Riemannian metric is
even if and only if it has degree zero. Similarly, we will say that a Riemannian metric is odd if and only
if it has degree one. Any Riemannian metric we consider will be either even or odd as we will only be
considering homogeneous metrics.\\
\indent Now we recall the definition of the warped product of Riemannian $\mathbb{Z}_2$-manifolds. For details, see the section 2.3 in \cite{BG}. Let $M_1\times M_2$
be the product of two $\mathbb{Z}_2$-manifolds $M_1$ and $M_2$. Let $(M_i,g_i)( i=1,2)$ be Riemannian $\mathbb{Z}_2$-manifolds whose Riemannian metric are of the same $\mathbb{Z}_2$-degree. Let $\mu\in C^{\infty}
(M_1)$ be a degree $0$ invertible global functions that is strictly positive, i.e. $\varepsilon_{M_1}(\mu)$ a strictly positive function on $|M_1|$ where $\varepsilon$ is simply "throwing away" the formal coordinates. Then the warped product is defined as
$$M_1\times_\mu M_2:=(M_1\times M_2,g:=\pi_1^*g_1+(\pi^*_1\mu)\pi^*_2g_2),$$
where $\pi_i: M_1\times M_2\rightarrow M_i~~(i=1,2)$ is the projection. By Proposition 4 in \cite{BG}, the warped product $M_1\times_\mu M_2$ is a Riemannian $\mathbb{Z}_2$-manifold.\\
\begin{defn}(Definition 9 in \cite{BG}) An affine connection on a $\mathbb{Z}_2$-manifold is a $\mathbb{Z}_2$-degree preserving map\\
$$\nabla:~~ {\rm Vect}(M)\times {\rm Vect}(M)\rightarrow {\rm Vect}(M);~~(X,Y)\mapsto \nabla_XY,$$
that satisfies the following\\
1) Bi-linearity $$\nabla_X(Y+Z)=\nabla_XY+\nabla_XZ;~~\nabla_{X+Y}Z=\nabla_XZ+\nabla_YZ,$$
2)$C^{\infty}(M)$-linearrity in the first argument
$$\nabla_{fX}Y=f\nabla_XY,$$
3)The Leibniz rule
$$\nabla_X(fY)=X(f)Y+(-1)^{|X||f|}f\nabla_XY,$$
for all homogeneous $X,Y,Z\in {\rm Vect}(M)$ and $f\in C^{\infty}(M)$.
\end{defn}
\begin{defn}(Definition 10 in \cite{BG})
 The torsion tensor of an affine connection \\
 $T_\nabla:~~{\rm Vect}(M)\otimes_{C^{\infty}(M)} {\rm Vect}(M)\rightarrow {\rm Vect}(M)$
 is
defined as
$$T_\nabla(X,Y):=\nabla_XY-(-1)^{|X||Y|}\nabla_YX-[X,Y],$$
for any (homogeneous) $X, Y \in {\rm Vect}(M)$. An affine connection is said to be symmetric if the torsion vanishes.
\end{defn}
\begin{defn}(Definition 11 in \cite{BG})
An affine connection on a Riemannian $\mathbb{Z}_2$-manifold $(M, g)$ is said to be metric compatible if
and only if
$$ X\left<Y,Z\right>_g=\left<\nabla_XY,Z\right>_g+(-1)^{|X||Y|}\left<Y,\nabla_XZ\right>_g,$$
for any $X, Y, Z\in  {\rm Vect}(M)$.
\end{defn}
\begin{thm}(Theorem 1 in \cite{BG})There is a unique symmetric (torsionless) and metric compatible
affine connection $\nabla^L$ on a Riemannian $\mathbb{Z}_2$-manifold $(M, g)$ which satisfies the Koszul formula
\begin{align}
2\left<\nabla^L_XY,Z\right>_g&=X\left<Y,Z\right>_g+\left<[X,Y],Z\right>_g\\
&+(-1)^{|X|(|Y|+|Z|)}(Y\left<Z,X\right>_g-\left<[Y,Z],X\right>_g)\notag\\
&-(-1)^{|Z|(|X|+|Y|)}(Z\left<X,Y\right>_g-\left<[Z,X],Y\right>_g),\notag
\end{align}
for all homogeneous $X,Y,Z\in {\rm Vect}(M)$.
\end{thm}
\begin{defn}(Definition 13 in \cite{BG})
The Riemannian curvature tensor of an affine connection
$$R_\nabla:~~{\rm Vect}(M)\otimes_{C^{\infty}(M)} {\rm Vect}(M)\otimes_{C^{\infty}(M)} {\rm Vect}(M)\rightarrow {\rm Vect}(M)$$
is defined as
\begin{equation}
R_\nabla(X, Y )Z =\nabla_X\nabla_Y-(-1)^{|X||Y|}\nabla_Y\nabla_X-\nabla_{[X,Y]}Z,
\end{equation}
for all $X, Y$ and $Z \in {\rm Vect}(M)$.
\end{defn}
Directly from the definition it is clear that
\begin{equation}
R_\nabla(X, Y )Z =-(-1)^{|X||Y|}R_\nabla(Y,X)Z,
\end{equation}
for all $X, Y$ and $Z \in {\rm Vect}(M)$.
\begin{defn}(Definition 14 in \cite{BG})
 The Ricci curvature tensor of an affine connection is the symmetric rank-$2$ covariant tensor
defined as
 \begin{equation}
Ric_\nabla(X, Y ):=(-1)^{|\partial_{x^I}|(|\partial_{x^I}|+|X|+|Y|)}\frac{1}{2}\left[R_\nabla(\partial_{x^I},X)Y+(-1)^{|X||Y|}R_\nabla(\partial_{x^I},Y)X\right]^I,
\end{equation}
where $X,Y\in {\rm Vect}(M)$ and $[ ~~ ~~]^I$ denotes the coefficient of $\partial_{x^I}$ and $\partial_{x^I}$ is the natural frame of $\mathcal{T}M$.
\end{defn}
\begin{defn}(Definition 16 in \cite{BG})
Let $f \in C^{\infty}(M)$ be an arbitrary function on a Riemannian $\mathbb{Z}_2$-manifold $(M, g)$. The gradient
of $f$ is the unique vector field ${\rm grad}_gf$ such that
 \begin{equation}
X(f)=(-1)^{|f||g|}\left<X,{\rm grad}_gf\right>_g,
\end{equation}
for all $X \in {\rm Vect}(M)$.
\end{defn}
\begin{defn}(Definition 17 in \cite{BG})
Let $(M, g)$ be a Riemannian $\mathbb{Z}_2$-manifold
and let $\nabla^L$ be the associated Levi-Civita connection.
The covariant divergence is the map ${\rm Div}_L:{\rm Vect}(M)\rightarrow C^{\infty}(M)$, given by
\begin{equation}
{\rm Div}_L(X)=(-1)^{|\partial_{x^I}|(|\partial_{x^I}|+|X|)}(\nabla_{\partial_{x^I}}X)^I,
\end{equation}
for any arbitrary $X \in {\rm Vect}(M)$.
\end{defn}
\begin{defn}(Definition 18 in \cite{BG}) Let $(M, g)$ be a Riemannian $\mathbb{Z}_2$-manifold
and let $\nabla^L$ be the associated Levi-Civita connection.
The connection Laplacian (acting on functions) is the differential operator of  $\mathbb{Z}_2$-degree $|g|$ defined as
\begin{equation}
\triangle_g(f)={\rm Div}_L({\rm grad}_gf),
\end{equation}
for any and all $f \in C^{\infty}(M).$
\end{defn}
\begin{defn}Let $(M, g)$ be a Riemannian $\mathbb{Z}_2$-manifold and $P\in {\rm Vect}(M)$ which satisfied $|g|+|P|=0$
and the semi-symmetric metric connection ${\nabla}$ on $(M, g)$
is given by
\begin{equation}
{\nabla}_XY=\nabla^L_XY+X\cdot g(Y,P)-g(X,Y)P=\nabla^L_XY+(-1)^{|X||Y|}g(Y,P)X-g(X,Y)P,
\end{equation}
for any homogenous $X,Y\in {\rm Vect}(M)$ and we define $\nabla_{X_1+X_2}Y=\nabla_{X_1}Y+\nabla_{X_2}Y;~~\nabla_{X}(Y_1+Y_2)=\nabla_XY_1+\nabla_XY_2,$ for the general $X=X_1+X_2$ and $Y=Y_1+Y_2$. Here
$X\cdot f=(-1)^{|X||f|}fX$ for $f\in C^{\infty}(M)$.
\end{defn}
We can verify that ${\nabla}_XY$ satisfies the definition 2.4, then ${\nabla}_XY$ is an affine connection. By Definition 2.5, we get
\begin{equation}
T_\nabla(X,Y)=X\cdot g(Y,P)-(-1)^{|X||Y|}Y\cdot g(X,P),
\end{equation}
In this case, we call that ${\nabla}_XY$ is a semi-symmetric connection. By Definition 2.6 and (2.8) and $\nabla^L$ preserving the metric and
 $|g|+|P|=0$, we get
\begin{align}
&\left<\nabla_XY,Z\right>_g+(-1)^{|X||Y|}\left<Y,\nabla_XZ\right>_g\\
&=X\left<Y,Z\right>_g+\left<X\cdot g(Y,P),Z\right>_g-\left<g(X,Y)P,Z\right>_g\notag\\
&+(-1)^{|X||Y|}[\left<Y,X\cdot g(Z,P)\right>_g-\left<Y,g(X,Z)P\right>_g]\notag\\
&=X\left<Y,Z\right>_g+(-1)^{|g(Y,P)||X|}g(Y,P)g(X,Z)-g(X,Y)g(P,Z)\notag\\
&+(-1)^{|X||Y|}(-1)^{|g(Z,P)||X|}g(Y,g(Z,P)X)-(-1)^{|X||Y|}(-1)^{|g(X,Z)||Y|}g(X,Z)g(Y,P),\notag\\
&=X\left<Y,Z\right>_g.\notag
\end{align}
So $\nabla$ preserves the metric.

\begin{thm}There is a unique metric compatible
affine connection $\nabla$ on a Riemannian $\mathbb{Z}_2$-manifold $(M, g)$ which satisfies (2.9).
\end{thm}
\begin{proof}By (2.9) and (2.10), we know that the semi-symmetric metric connection $\nabla$ satisfies the conditions in Theorem 2.14. We only prove
the uniqueness. Let $\nabla^1$ be a connection which satisfies the conditions in Theorem 2.14. Let $\nabla^1_XY=\nabla^L_XY+H(X,Y).$
then \begin{equation}
H(fX,Y)=fH(X,Y),~~H(X,fY)=(-1)^{|f||X|}H(X,Y).
\end{equation}
By $\nabla^1$ and $\nabla^L$ preserving the metric, we get
\begin{equation}
g(H(X,Y),Z)+(-1)^{|X||Y|}g(Y,H(X,Z))=0.
\end{equation}
By $\nabla^L$ having no torsion, we have
\begin{equation}
T_{\nabla^1}(X,Y)=H(X,Y)-(-1)^{|X||Y|}H(Y,X).
\end{equation}
By (2.12) and (2.13) and $|H|=0$, we have
\begin{align}
&g(T_{\nabla^1}(X,Y),Z)+(-1)^{|Z|(|X|+|Y|)}g(T_{\nabla^1}(Z,X),Y)\\
&+(-1)^{|X||Y|}(-1)^{|Z|(|X|+|Y|)}g(T_{\nabla^1}(Z,Y),X)=2g(H(X,Y),Z).\notag
\end{align}
By (2.9) and (2.14), we get
\begin{equation}
2g(H(X,Y),Z)=-2g(X,Y)g(Y,Z)+2g(X\cdot g(Y,P),Z).
\end{equation}
Then $H(X,Y)=X\cdot g(Y,P)-g(X,Y)P$ and $\nabla^1$ is $\nabla$ in the definition 2.13.
\end{proof}
\indent Let $\pi$ be a one form defined by $\pi(Z):=g(Z,P)$, then $|\pi|=0$. We have

\begin{prop}The following equality holds
\begin{align}
&R_\nabla(X,Y)Z=R^L(X,Y)Z+(-1)^{(|X|+|Y|)|Z|}g(Z,\nabla^L_XP)Y\\
&-(-1)^{|X||Y|}(-1)^{(|X|+|Y|)|Z|}g(Z,\nabla^L_YP)X
-(-1)^{|g(Y,Z)||X|}g(Y,Z)\nabla^L_XP\notag\\
&+(-1)^{|X||Y|}(-1)^{|g(X,Z)||Y|}g(X,Z)\nabla^L_YP
+(-1)^{|X|(|Y|+|Z|)}(-1)^{|Y||Z|}\pi(Z)\pi(Y)X\notag\\
&-(-1)^{|X|(|Y|+|Z|)}g(Y,Z)\pi(P)X
-(-1)^{(|X|+|Y|)|Z|}\pi(Z)\pi(X)Y\notag\\
&+(-1)^{|Y||Z|)}g(X,Z)\pi(P)Y
+(-1)^{|X||g(Y,Z)|}g(Y,Z)\pi(X)P\notag\\
&-(-1)^{|X||Y|}(-1)^{|Y||g(X,Z)|}g(X,Z)\pi(Y)P.\notag
\end{align}
\end{prop}
\begin{proof}
By the Leibniz rule and $\nabla^L$ preserving metric, we have
\begin{equation}
\nabla^L_X(\pi(Z)Y)=\pi(\nabla^L_XZ)Y+(-1)^{|X||Z|}g(Z,\nabla^L_XP)Y+(-1)^{|X||Z|}\pi(Z)\nabla^L_XY,
\end{equation}
\begin{equation}
\nabla^L_X(g(Y,Z)P)=g(\nabla^L_XY,Z)P+(-1)^{|X||Y|}g(Y,\nabla^L_XZ)P+(-1)^{|g(Y,Z)||X|}g(Y,Z)\nabla^L_XP.
\end{equation}
By (2.2),(2.8),(2.17) and (2.18) and some computations, we can get Proposition 2.15.
\end{proof}

\section{Super warped products with a semi-symmetric metric connection}
Let $(M_1=M\times_\mu N,g_\mu=\pi^*_1 g_1+\pi^*_1(\mu)\pi_2^*g_2)$ be the super warped product with $|g|=|g_1|=|g_2|$ and $|\mu|=0$. For simplicity, we assume that $\mu=h^2$ with $|h|=0$. Let $\nabla^{L,\mu}$ be the Levi-Civita connection on $(M_1,g_\mu)$ and $\nabla^{L,M}$ (resp. $\nabla^{L,N}$) be the Levi-Civita connection on $(M,g_1)$ (resp.$(N,g_2)$).
\begin{lem}
For $X,Y\in{\rm Vect}(M)$ and $U,W\in {\rm Vect}(N)$, we have
\begin{align}
&(1)\nabla^{L,\mu}_XY=\nabla^{L,M}_XY,~~(2)\nabla^{L,\mu}_XU=\frac{X(h)}{h}U,\\
&(3)\nabla^{L,\mu}_UX=(-1)^{|U||X|}\frac{X(h)}{h}U,~~(4)\nabla^{L,\mu}_UW=-hg_2(U,W){\rm grad}_{g_1}h+\nabla^{L,N}_UW.\notag
\end{align}
\end{lem}
\begin{proof}
By (2.1) and $[X,V]=0$, we have $g_\mu(\nabla^{L,\mu}_XY,Z)=g_1(\nabla^{L,M}_XY,Z)$ and $g_\mu(\nabla^{L,\mu}_XY,V)=0$, so (1) holds.
Similarly, we have $g_\mu(\nabla^{L,\mu}_XU,Y)=0$ and $2g_\mu(\nabla^{L,\mu}_XU,V)=\frac{X(\mu)}{\mu}g_\mu(U,V)$, so (2) holds by $\mu=h^2$.
By $\nabla^{L,\mu}$ having no torsion and (2), we have (3). By (2.1) and (2.5), we have
\begin{align}
&2g_\mu(\nabla^{L,\mu}_UW,X)=-(-1)^{|X|(|U|+|W|)}X(\mu)g_2(U,W)\\
&=-(-1)^{|X|(|U|+|W|)}g_1(X,{\rm grad}_{g_1}(\mu))g_2(U,W)
=-g_\mu(g_2(U,W){\rm grad}_{g_1}(\mu),X),\notag
\end{align}
and $g_\mu(\nabla^{L,\mu}_UW,U_1)=g_\mu(\nabla^{L,N}_UW,U_1)$ for $U_1\in {\rm Vect}(N)$, so (4) holds.
\end{proof}

Let $R^{L,\mu}$ denote the curvature tensor of the Levi-Civita connection on $(M_1,g_\mu)$. Let $R^{L,M}$ (resp. $R^{L,N}$) be the curvature tensor of the Levi-Civita connection on $(M,g_1)$ (resp.$(N,g_2)$). Let $H^h_M(X,Y):=XY(h)-\nabla^{L,M}_XY(h)$, then
$H^h_M(fX,Y)=fH^h_M(X,Y)$ and $H^h_M(X,fY)=(-1)^{|f||X|}fH^h_M(X,Y)$. $H^h_M$ is a $(0,2)$ tensor.
\begin{prop}
For $X,Y,Z\in{\rm Vect}(M)$ and $U,V,W\in {\rm Vect}(N)$, we have
\begin{align}
&(1)R^{L,\mu}(X,Y)Z=R^{L,M}(X,Y)Z,~~(2)R^{L,\mu}(V,X)Y=-(-1)^{|V|(|X|+|Y|)}\frac{H^h_M(X,Y)}{h}V,\\
&(3)R^{L,\mu}(X,Y)V=0,~~(4)R^{L,\mu}(V,W)X=0,\notag\\
&(5)R^{L,\mu}(X,V)W=-(-1)^{|X|(|V|+|W|+|g|)}\frac{g_\mu(V,W)}{h}\nabla^{L,M}_X({\rm grad}_{g_1}h),\notag\\
&(6)R^{L,\mu}(V,W)U=R^{L,N}(V,W)U-(-1)^{|V|(|W|+|U|)}g_2(W,U)({\rm grad}_{g_1}h)(h)V\notag\\
&+(-1)^{|W||U|}g_2(V,U)({\rm grad}_{g_1}h)(h)W.\notag
\end{align}
\end{prop}
\begin{proof}(1) (1) comes from (1) in Lemma 3.1 and (2.2).\\
\indent (2)By Lemma 3.1 and the Leibniz rule, we have
\begin{equation}
\nabla^{L,\mu}_V\nabla^{L,\mu}_XY=(-1)^{|V|(|X|+|Y|)}\frac{\nabla^{L,M}_XY(h)}{h}V,~~-\nabla^{L,\mu}_X\nabla^{L,\mu}_VY=-(-1)^{|V|(|X|+|Y|)}\frac{XY(h)}{h}V
\end{equation}
By (2.2) and (3.4) and $[V,X]=0$ and the definition of $H^h_M(X,Y)$, we get (2).\\
\indent (3) By Lemma 3.1, we have $\nabla^{L,\mu}_X\nabla^{L,\mu}_YV=\frac{XY(h)}{h}V$. So by (2.2) and the definition of $[X,Y]$, we get (3).\\
\indent(4)By Lemma 3.1, we have
\begin{equation}
\nabla^{L,\mu}_V\nabla^{L,\mu}_WX=(-1)^{|W||X|}[V(\frac{X(h)}{h})W+(-1)^{|V||X|}\frac{X(h)}{h}\nabla^{L,\mu}_VW].
\end{equation}
By (2.2) and $V(\frac{X(h)}{h})=0$ and $\nabla^{L,\mu}$ having no torsion, we get (4).\\
\indent (5)For $W_1\in {\rm Vect}(N)$, we have by (4.12) in \cite{Go} and (4)
\begin{equation}
g_\mu(R^{L,\mu}(X,V)W,W_1)=(-1)^{(|X|+|V|)(|W|+|W_1|)}g_\mu(R^{L,\mu}(W,W_1)X,V)=0
\end{equation}
By Proposition 9 in \cite{BG} and (2), we have
\begin{align}
&g_\mu(R^{L,\mu}(X,V)W,Y)=-(-1)^{|W||Y|}g_\mu(R^{L,\mu}(X,V)Y,W)\\
&=-(-1)^{(|W|+|V|)|Y|}\frac{H^h_M(X,Y)}{h}g_\mu(V,W).\notag
\end{align}
By the definition of ${\rm grad}_{g_1}(h)$ and $\nabla^{L,M}$ preserving the metric, we can get
\begin{equation}
g_1(\nabla^{L,M}_X({\rm grad}_{g_1}h),Y)=(-1)^{|Y||g_1|}H^h_M(X,Y).
\end{equation}
So
\begin{equation}
g_\mu(R^{L,\mu}(X,V)W,Y)=-(-1)^{|X|(|V|+|W|+|g|)}g_\mu(\frac{g_\mu(V,W)}{h}\nabla^{L,M}_X({\rm grad}_{g_1}h),Y).
\end{equation}
By (3.6) and (3.9), we get (5).\\
\indent (6) By (4), we have
\begin{equation}
g_\mu(R^{L,\mu}(V,W)U,X)=-(-1)^{|X||U|}g_\mu(R^{L,\mu}(V,W)X,U)=0.
\end{equation}
By Lemma 3.1 and the Leibniz rule, we have
\begin{equation}
g_\mu(\nabla^{L,\mu}_V\nabla^{L,\mu}_WU,W_1)=g_\mu(\nabla^{L,N}_V\nabla^{L,N}_WU,W_1)-(-1)^{|V|(|U|+|W|)}g_\mu(W,U)\frac{({\rm grad}_{g_1}h)(h)}{h^2}g_\mu(V,W_1).
\end{equation}
Then by (2.2) and (3.11), we have
\begin{align}
&g_\mu(R^{L,\mu}(V,W)U,W_1)=g_\mu(R^{L,N}(V,W)U,W_1)\\
&-(-1)^{|V|(|U|+|W|)}g_\mu(W,U)\frac{({\rm grad}_{g_1}h)(h)}{h^2}g_\mu(V,W_1)\notag\\
&+(-1)^{|U||W|}g_\mu(V,U)\frac{({\rm grad}_{g_1}h)(h)}{h^2}g_\mu(W,W_1).\notag
\end{align}
By (3.10) and (3.12), we get (6).
\end{proof}
For $\overline{X},\overline{Y},{P}\in {\rm Vect}(M_1)$, we define
\begin{equation}
\nabla^\mu_{\overline{X}}\overline{Y}=\nabla^{L,\mu}_{\overline{X}}\overline{Y}+\overline{X}\cdot g_\mu(\overline{Y},{P})-g_\mu(\overline{X},\overline{Y}){P}.
\end{equation}
For ${X},{Y},{P}\in {\rm Vect}(M)$, we define
\begin{equation}
\nabla^{M}_{X}{Y}=\nabla^{L,M}_{X}{Y}+{X}\cdot g_1({Y},{P})-g_1({X},{Y}){P}.
\end{equation}
By Lemma 3.1 and (3.13), (3.14), we have
\begin{lem}
For $X,Y,P\in{\rm Vect}(M)$ and $U,W\in {\rm Vect}(N)$ and $\pi(X)=g_1(X,P)$, we have
\begin{align}
&(1)\nabla^{\mu}_XY=\nabla^{M}_XY,~~(2)\nabla^{\mu}_XU=\frac{X(h)}{h}U,\\
&(3)\nabla^{\mu}_UX=(-1)^{|U||X|}[\frac{X(h)}{h}+\pi(X)]U,\notag\\
&(4)\nabla^{\mu}_UW=-hg_2(U,W){\rm grad}_{g_1}h+\nabla^{L,N}_UW-g_\mu(U,W)P.\notag
\end{align}
\end{lem}
\begin{lem}
For $X,Y\in{\rm Vect}(M)$ and $U,W,P\in {\rm Vect}(N)$, we have
\begin{align}
&(1)\nabla^{\mu}_XY=\nabla^{L,M}_XY-g_1(X,Y)P,~~(2)\nabla^{\mu}_XU=\frac{X(h)}{h}U+X\cdot g_\mu(U,P),\\
&(3)\nabla^{\mu}_UX=(-1)^{|U||X|}\frac{X(h)}{h}U,\notag\\
&(4)\nabla^{\mu}_UW=-hg_2(U,W){\rm grad}_{g_1}h+\nabla^{L,N}_UW+U\cdot g_\mu(W,P)-g_\mu(U,W)P.\notag
\end{align}
\end{lem}
By Proposition 2.15 and Proposition 3.2 and Lemma 3.1, we get by some computations
\begin{prop}
For $X,Y,Z,P\in{\rm Vect}(M)$ and $U,V,W\in {\rm Vect}(N)$, we have
\begin{align}
(1)R_{\nabla^\mu}(X,Y)Z&=R_{\nabla^M}(X,Y)Z,\\
(2)R_{\nabla^\mu}(V,X)Y&=-(-1)^{|V|(|X|+|Y|)}\left[\frac{H^h_M(X,Y)}{h}+(-1)^{|X||Y|}g_1(Y,\nabla^{L,M}_XP)\right.\notag\\
&\left.+g_1(X,Y)\frac{P(h)}{h}+g_1(X,Y)\pi(P)-\pi(X)\pi(Y)\right]V,\notag\\
(3)R_{\nabla^\mu}(X,Y)V&=0,~~(4)R_{\nabla^\mu}(V,W)X=0,\notag\\
(5)R_{\nabla^\mu}(X,V)W&=-(-1)^{|X|(|V|+|W|+|g|)}{g_\mu(V,W)}\left[\frac{\nabla^{L,M}_X({\rm grad}_{g_1}h)}{h}\right.\notag\\
&\left.+(-1)^{(|X|+|P|)|g|}\frac{P(h)}{h}X+\nabla^{L,M}_XP+X\cdot g_1(P,P)-\pi(X)P\right],\notag\\
{\rm When}~~|g|=|P|=0&,~~{\rm then}\notag\\
R_{\nabla^\mu}(X,V)W&=-(-1)^{|X|(|V|+|W|)}{g_\mu(V,W)}\left[\frac{\nabla^{L,M}_X({\rm grad}_{g_1}h)}{h}\right.\notag\\
&\left.+\frac{P(h)}{h}X+\nabla^{L,M}_XP+X\cdot g_1(P,P)-\pi(X)P\right],\notag\\
(6)R_{\nabla^\mu}(U,V)W&=R^{L,N}(U,V)W+\left[(-1)^{|g|(|W|+|g|)}\frac{({\rm grad}_{g_1}h)(h)}{h^2}\right.\notag\\
&\left.+(-1)^{|P|(|W|+|g|)}\frac{P(h)}{h}
+(-1)^{|P|(|W|+|g|)}\pi(P)\right]\notag\\
&\cdot\left[(-1)^{|V||W|}(-1)^{|P||U|}g_\mu(U,W)V-(-1)^{|U|(|V|+|W|)}(-1)^{|P||V|}g_\mu(V,W)U\right]\notag\\
&+(-1)^{|V||W|}g_\mu(U,W)\frac{P(h)}{h}V-(-1)^{|U|(|V|+|W|)}g_\mu(V,W)\frac{P(h)}{h}U,\notag\\
{\rm When}~~|g|=|P|=0&,~~{\rm then}\notag\\
R_{\nabla^\mu}(U,V)W&=R^{L,N}(U,V)W+\left[\frac{({\rm grad}_{g_1}h)(h)}{h^2}
+2\frac{P(h)}{h}+\pi(P)\right]\notag\\
&\cdot\left[(-1)^{|V||W|}g_\mu(U,W)V-(-1)^{|U|(|V|+|W|)}g_\mu(V,W)U\right].\notag
\end{align}
\end{prop}
Similarly, we have
\begin{prop}
For $X,Y,Z\in{\rm Vect}(M)$ and $U,V,W,P\in {\rm Vect}(N)$, we have
\begin{align}
(1)R_{\nabla^\mu}(X,Y)Z&=R_{\nabla^{L,M}}(X,Y)Z+(-1)^{|X||Y|}\frac{Y(h)}{h}g_\mu(X,Z)P\\
&-\frac{X(h)}{h}g_\mu(Y,Z)P-(-1)^{|X|(|Y|+|Z|)}g_\mu(Y,Z)\pi(P)X\notag\\
&+(-1)^{|Y||Z|}g_\mu(X,Z)\pi(P)Y,\notag\\
(2)R_{\nabla^\mu}(V,X)Y&=-(-1)^{|V|(|X|+|Y|)}\frac{H^h_M(X,Y)}{h}V\notag\\
&-(-1)^{|X||Y|}hg_2(V,P)g_1(Y,{\rm grad}_{g_1}h)X\notag\\
&-(-1)^{|g(X,Y)||V|}g_1(X,Y)[\nabla^{L,N}_VP-hg_2(V,P){\rm grad}_{g_1}h]\notag\\
&-(-1)^{|V|(|X|+|Y|)}g_1(X,Y)\pi(P)V+(-1)^{|g(X,Y)||V|}g_1(X,Y)\pi(V)P,\notag\\
(3)R_{\nabla^\mu}(X,Y)V&=(-1)^{(|X|+|Y|)|V|}\pi(V)[\frac{X(h)}{h}Y-(-1)^{|X||Y|}\frac{Y(h)}{h}X],\notag\\
(4)R_{\nabla^\mu}(V,W)X&=-(-1)^{|X||W|}hg_2(V,P)g_1(X,{\rm grad}_{g_1}h)W\notag\\
&+(-1)^{|V||W|}(-1)^{|X||V|}hg_2(W,P)g_1(X,{\rm grad}_{g_1}h)V,\notag\\
(5)R_{\nabla^\mu}(X,V)W&=-(-1)^{|X|(|V|+|W|+|g|)}\frac{{g_\mu(V,W)}}{h}{\nabla^{L,M}_X({\rm grad}_{g_1}h)}\notag\\
&+(-1)^{(|X|+|V|)|W|}[(-1)^{|X||W|}\frac{X(h)}{h}g_\mu(W,P)V-(-1)^{|X||V|}g_\mu(W,\nabla^{L,N}_VP)X]\notag\\
&-(-1)^{|g_\mu(V,W)||X|}g_\mu(V,W)\frac{X(h)}{h}P-(-1)^{|X|(|V|+|W|)}g_\mu(V,W)\pi(P)X\notag\\
&+(-1)^{(|X|+|V|)|W|}(-1)^{|X||V|}\pi(W)\pi(V)X,\notag\\
(6)R_{\nabla^\mu}(U,V)W&=R^{L,N}(U,V)W-(-1)^{|U|(|V|+|W|)}g_2(V,W)({\rm grad}_{g_1}h)(h)U\notag\\
&+(-1)^{|V||W|}g_2(U,W)({\rm grad}_{g_1}h)(h)V\notag\\
&+(-1)^{(|U|+|V|)|W|}[g_\mu(W,\nabla^{L,N}_UP)V-(-1)^{|U||V|}g_\mu(W,\nabla^{L,N}_VP)U]\notag\\
&+(-1)^{|U||V|}(-1)^{|g(U,W)||V|}g_\mu(U,W)[\nabla^{L,N}_VP-hg_2(V,P){\rm grad}_{g_1}h]\notag\\
&-(-1)^{|g(V,W)||U|}g_\mu(V,W)[\nabla^{L,N}_UP-hg_2(U,P){\rm grad}_{g_1}h]\notag\\
&-(-1)^{|U|(|V|+|W|)}g_\mu(V,W)\pi(P)U
+(-1)^{|V||W|}g_\mu(U,W)\pi(P)V\notag\\
&+(-1)^{|U||g_\mu(V,W)|}g_\mu(V,W)\pi(U)P
-(-1)^{|U||V|}(-1)^{|g(U,W)||V|}g_\mu(U,W)\pi(V)P\notag\\
&+(-1)^{(|U|+|V|)|W|}\pi(W)[(-1)^{|U||V|}\pi(V)U-\pi(U)V].\notag
\end{align}
\end{prop}
In the following, we compute the Ricci tensor of $M_1$. Let $M$ (resp. $N$) have the $(p,m)$ (resp. $(q,n)$) dimension. Let $\partial_{x^I}=\{\partial_{x^a},\partial_{\xi^A}\}$ (resp.
$\partial_{y^J}=\{\partial_{y^b},\partial_{\eta^B}\}$) denote
the natural tangent frames on $M$ (resp. $N$). Let ${\rm Ric}^{L,\mu}$ (resp. ${\rm Ric}^{L,M}$, ${\rm Ric}^{L,N}$)  denote the Ricci tensor of $(M_1,g_\mu)$ (resp. $(M,g_1)$, $(N,g_2)$).  Then by (2.4), (2.7) and (3.3), we have

\begin{prop}
The following equalities holds
\begin{align}
(1){\rm Ric}^{L,\mu}(\partial_{x^I},\partial_{x^K})&={\rm Ric}^{L,M}(\partial_{x^I},\partial_{x^K})-\frac{(q-n)}{h}H^h_M(\partial_{x^I},\partial_{x^K}),\\
(2){\rm Ric}^{L,\mu}(\partial_{x^I},\partial_{y^J})&={\rm Ric}^{L,\mu}(\partial_{y^J},\partial_{x^I})=0,\notag\\
(3){\rm Ric}^{L,\mu}(\partial_{y^L},\partial_{y^J})&={\rm Ric}^{L,N}(\partial_{y^L},\partial_{y^J})-g_\mu(\partial_{y^L},\partial_{y^J})[\frac{\triangle^L_{g_1}(h)}{h}+(q-n-1)\frac{({\rm grad}_{g_1}h)(h)}{h^2}].\notag
\end{align}
\end{prop}
Let ${\rm Ric}^{\nabla^\mu}$ (resp. ${\rm Ric}^{\nabla^M}$) denote the Ricci tensor of $(M_1,\nabla^\mu,g_\mu)$ (resp. $(M,g_1,\nabla^M)$). Then by Proposition 3.5 and (2.4), (2.6), we have
\begin{prop}
The following equalities holds
\begin{align}
(1){\rm Ric}^{\nabla^\mu}(\partial_{x^I},\partial_{x^K})&={\rm Ric}^{\nabla^M}(\partial_{x^I},\partial_{x^K})-(q-n)\left[
\frac{H^h_M(\partial_{x^I},\partial_{x^K})}{h}\right.\\
&+\frac{1}{2}(-1)^{|\partial_{x^I}||\partial_{x^K}|}g_1(\partial_{x^K},\nabla^{L,M}_{\partial_{x^I}}P)
+\frac{1}{2}g_1(\partial_{x^I},\nabla^{L,M}_{\partial_{x^K}}P)\notag\\
&\left.+g_1(\partial_{x^I},\partial_{x^K})\frac{P(h)}{h}
+g_1(\partial_{x^I},\partial_{x^K})\pi(P)-\pi(\partial_{x^I})\pi(\partial_{x^K})\right],\notag\\
(2){\rm Ric}^{\nabla^\mu}(\partial_{x^I},\partial_{y^J})&={\rm Ric}^{\nabla^\mu}(\partial_{y^J},\partial_{x^I})=0,\notag\\
{\rm When}~~|g|=|P|=0,~~{\rm then}&\notag\\
(3){\rm Ric}^{\nabla^\mu}(\partial_{y^L},\partial_{y^J})&={\rm Ric}^{L,N}(\partial_{y^L},\partial_{y^J})-g_\mu(\partial_{y^L},\partial_{y^J})
[\frac{\triangle^L_{g_1}(h)}{h}+{\rm Div}^M_L(P)\notag\\
&+(q-n-1)\frac{({\rm grad}_{g_1}h)(h)}{h^2}
+(2q+p-m-2n-2)\frac{P(h)}{h}\notag\\
&+(p+q-m-n-2)\pi(P)].\notag
\end{align}
\end{prop}
\section{Special super warped products with a semi-symmetric metric connection}
\indent In this section, we construct an Einstein super warped product with a semi-symmetric metric connection. Let $(N^{(q,n)},g_2)$ be a super Riemannian manifold and $\mathbb{R}^{(1,0)}$ be the real line. Let
$h(t)$ and $\mu(t)=h(t)^2$ be non-zero functions for $t\in \mathbb{R}$. Let $|g_2|=0$. We consider the super Riemannian manifold $M_1=\mathbb{R}^{(1,0)}\times _\mu N^{(q,n)}$ and $g_\mu=-dt\otimes dt+h^2g_2$.
Let $P=\partial_t$. Then $R_{\nabla^{\mathbb{R}}}(\partial_t,\partial_t)\partial_t=0$  and ${\rm Ric}^{\nabla^{\mathbb{R}}}(\partial_{t},\partial_{t})=0$. We have $H^h_M(\partial_{t},\partial_{t})=h''$, ${\rm grad}_{g_1}(h)=-h'\partial_t$ and $\triangle_{g_1}^L(h)=-h''$. By (3.20), we have
\begin{prop}
The following equalities holds
\begin{align}
(1){\rm Ric}^{\nabla^\mu}(\partial_{t},\partial_{t})&=-(q-n)(\frac{h''}{h}-\frac{h'}{h}),\\
(2){\rm Ric}^{\nabla^\mu}(\partial_{t},\partial_{y^J})&={\rm Ric}^{\nabla^\mu}(\partial_{y^J},\partial_{t})=0,\notag\\
(3){\rm Ric}^{\nabla^\mu}(\partial_{y^L},\partial_{y^J})&={\rm Ric}^{L,N}(\partial_{y^L},\partial_{y^J})-g_\mu(\partial_{y^L},\partial_{y^J})\notag\\
&\cdot[-\frac{h''}{h}-(q-n-1)\frac{(h')^2}{h^2}+(2q-2n-1)\frac{h'}{h}-(q-n-1)].
\notag
\end{align}
\end{prop}
We call that $(M_1,g_\mu,\nabla^\mu)$ is Einstein if
\begin{equation}
{\rm Ric}^{\nabla^\mu}(\overline{X},\overline{Y})=\lambda g_\mu(\overline{X},\overline{Y})£¬
\end{equation}
for $\overline{X},\overline{Y}\in {\rm Vect}(M_1)$ and a constant $\lambda$. As in the ordinary warped product case (see Theorem 15 in \cite{W1}), by (4.1) and (4.2), we have

\begin{thm} Let $M_1=\mathbb{R}^{(1,0)}\times _\mu N^{(q,n)}$ and $g_\mu=-dt\otimes dt+h^2g_2$ and $P=\partial_t$.Then
    $(M_1,g_\mu,\nabla^\mu)$ is Einstein with the Einstein constant
    $\lambda$ if and only if the following conditions are satisfied\\
    \noindent (1) $( N^{(q,n)},\nabla^{L,N})$ is Einstein with the Einstein constant $c_0$.\\
\noindent (2)
\begin{equation}
(q-n)(\frac{h''}{h}-\frac{h'}{h})=\lambda,
\end{equation}
(3)
\begin{equation}
\lambda h^2-h''h-(q-n-1)(h')^2+(2q-2n-1)hh'-(q-n-1)h^2=c_0.
\end{equation}
\end{thm}
By Theorem 4.2, similar to the ordinary warped product case (see Theorem 25 and Theorem 26 in \cite{W1}), we have

\begin{thm}
Let $M_1=\mathbb{R}^{(1,0)}\times _\mu N^{(q,n)}$ and $g_\mu=-dt\otimes dt+h^2g_2$ and $P=\partial_t$. We assume that $q-n=1$. Then
    $(M_1,g_\mu,\nabla^\mu)$ is Einstein with the Einstein constant
    $-\lambda_0$ if and only if the following conditions are satisfied\\
    \noindent (1) $( N^{(q,n)},\nabla^{L,N})$ is Einstein with the Einstein constant $c_0=0$.\\
\noindent (2-1) $\lambda_0<\frac{1}{4},~~f(t)=c_1e^{\frac{1+\sqrt{1-4\lambda_0}}{2}t}+c_2e^{\frac{1-\sqrt{1-4\lambda_0}}{2}t},$\\
 \noindent (2-2)
$\lambda_0=\frac{1}{4},~~f(t)=c_1e^{\frac{1}{2}t}+c_2te^{\frac{1}{2}t},$\\
\noindent (2-3) $\lambda_0>\frac{1}{4},~~f(t)=c_1e^{\frac{1}{2}t}{\rm
cos}\left(\frac{\sqrt{4\lambda_0-1}}{2}t\right)+ c_2e^{\frac{1}{2}t}{\rm sin}\left(\frac{\sqrt{4\lambda_0-1}}{2}t\right),$\\
\end{thm}
\begin{prop}
Let $M_1=\mathbb{R}^{(1,0)}\times _\mu N^{(q,n)}$ and $g_\mu=-dt\otimes dt+h^2g_2$and $P=\partial_t$. We assume that $q-n=0$. Then
    $(M_1,g_\mu,\nabla^\mu)$ is Einstein with the Einstein constant
    $-\lambda_0$ if and only if the following conditions are satisfied\\
    \noindent (1) $( N^{(q,n)},\nabla^{L,N})$ is Einstein with the Einstein constant $-c_0$.\\
\noindent (2)$\lambda_0=0,$\\
\noindent (3) $c_0-hh''+h'^2+h^2-hh'=0$\\
\end{prop}
\begin{thm}
Let $M_1=\mathbb{R}^{(1,0)}\times _\mu N^{(q,n)}$ and $g_\mu=-dt\otimes dt+h^2g_2$and $P=\partial_t$. We assume that $q-n\neq 0,1$. Then
    $(M_1,g_\mu,\nabla^\mu)$ is Einstein with the Einstein constant
    $-\lambda_0$ if and only if $\lambda_0=0$ and $h=c_1 e^t+c_2$ and
     $( N^{(q,n)},\nabla^{L,N})$ is Einstein with the Einstein constant $(q-n-1)c_2^2$.\\
\end{thm}
 \indent Nextly, we give another example. Let $M=\mathbb{R}^{(1,2)}$ with coordinates $(t,\xi,\eta)$ and $|t|=0,~~|\xi|=|\eta|=1$.
 We give a metric $g_1=-dt\otimes dt+d\xi\otimes d\eta-d\eta\otimes d\xi$ on $M$ i.e.
\begin{equation}
g_1(\partial_t,\partial_t)=-1,~~g_1(\partial_\xi,\partial_\eta)=-1,~~g_1(\partial_\eta,\partial_\xi)=1,~~g_1(\partial_{x^I},\partial_{x^K})=0,
\end{equation}
for the other pair $(\partial_{x^I},\partial_{x^K})$. Let $M_2=\mathbb{R}^{(1,2)}\times_\mu N^{(q,n)}$ and $g_\mu=g_1+h(t)^2g_2$ and $P=\partial_t$. By Proposition 7 in \cite{BG}, we have the Christoffel symbols $\Gamma^L_{JI}=0$, then
\begin{equation}
\nabla^{L,g_1}_{\partial_{x^J}}\partial_{x^K}=0,~~ R^{L,g_1}(X,Y)Z=0,~~{\rm Ric}^{L,g_1}(X,Y)=0.
\end{equation}
We have
\begin{equation}
H^h_M(\partial_{t},\partial_{t})=h'',~~H^h_M(\partial_{x^J},\partial_{x^K})=0,~~{\rm for~~ the~~ other~~ pair~~} (\partial_{x^I},\partial_{x^K}).
\end{equation}
\begin{equation}
 {\rm grad}_{g_1}(h)=-h'\partial_t,~~ \triangle_{g_1}^L(h)=-h''.
\end{equation}
By Proposition 3.7 and the Einstein condition, we have
\begin{thm}
Let $M_2=\mathbb{R}^{(1,2)}\times _\mu N^{(q,n)}$ and $g_\mu=g_1+h^2g_2$and $P=\partial_t$. Then
    $(M_2,g_\mu,\nabla^{L,\mu})$ is Einstein with the Einstein constant
    $\lambda$ if and only if one of the following conditions is satisfied\\
    \noindent (1) $\lambda=0$, $q=n$, $( N^{(q,n)},\nabla^{L,N})$ is Einstein with the Einstein constant $-c_0$ and $hh''-h'^2=c_0$.\\
\noindent (2) $\lambda=0$, $q-n-1=0$, $( N^{(q,n)},\nabla^{L,N})$ is Einstein with the Einstein constant $0$ and $h=c_1t+c_2$ where $c_1,c_2$ are constant.\\
\noindent (3) $\lambda=0$, $q-n-1\neq 0,-1$, $( N^{(q,n)},\nabla^{L,N})$ is Einstein with the Einstein constant $-c_0$ and
$h=\pm \sqrt{\frac{c_0}{q-n-1}}t+c_2$, $\frac{c_0}{q-n-1}\geq 0$.\\
\end{thm}
\indent By (2.16) and (4.6), we can get
\begin{align}
&{R}^{\nabla^{\mathbb{R}^{(1,2)}}}(\partial_{\xi},\partial_{\eta})\partial_{\xi}=
{R}^{\nabla^{\mathbb{R}^{(1,2)}}}(\partial_{\eta},\partial_{\xi})\partial_{\xi}=\partial_{\xi},\\
&{R}^{\nabla^{\mathbb{R}^{(1,2)}}}(\partial_{\xi},\partial_{\eta})\partial_{\eta}=
{R}^{\nabla^{\mathbb{R}^{(1,2)}}}(\partial_{\eta},\partial_{\xi})\partial_{\eta}
=-\partial_{\eta},\notag\\
&{R}^{\nabla^{\mathbb{R}^{(1,2)}}}(\partial_{\xi},\partial_{\xi})\partial_{\eta}=-2\partial_{\xi},~~
{R}^{\nabla^{\mathbb{R}^{(1,2)}}}(\partial_{\eta},\partial_{\eta})\partial_{\xi}=2\partial_{\eta},\notag\\
&{R}^{\nabla^{\mathbb{R}^{(1,2)}}}(\partial_{x^J},\partial_{x^K})\partial_{x^L}=0,\notag
\end{align}
for other pairs $(\partial_{x^J},\partial_{x^K},\partial_{x^L})$.
By (2.4) and (4.9), we have
\begin{align}
&{\rm Ric}^{\nabla^{\mathbb{R}^{(1,2)}}}(\partial_{\xi},\partial_{\eta})=-{\rm Ric}^{\nabla^{\mathbb{R}^{(1,2)}}}(\partial_{\eta},\partial_{\xi})=3,~~
{\rm Ric}^{\nabla^{\mathbb{R}^{(1,2)}}}(\partial_{x^J},\partial_{x^L})=0,
\end{align}
for other pairs $(\partial_{x^J},\partial_{x^L})$.\\
\indent If $(M_2,g_\mu,\nabla^{\mu})$ is Einstein with the Einstein constant
    $\lambda$, by (3.20) and (4.10), we have
 \begin{equation}
 (q-n)(\frac{h''}{h}-\frac{h'}{h})=\lambda,~~3-(q-n)(-\frac{h'}{h}+1)=-\lambda.
\end{equation}
 Solving (4.11), we get
  \begin{equation}
 h=c_1e^{{\pm\sqrt{1-\frac{3}{q-n}}}t},~~~\frac{\lambda+3}{q-n}=1\mp \sqrt{1-\frac{3}{q-n}},~~~1-\frac{3}{q-n}\geq 0.
\end{equation}
 By (3.20) (3) and the Einstein condition, we get
$( N^{(q,n)},\nabla^{L,N})$ is Einstein with the Einstein constant $c_0$ and
\begin{equation}
\lambda h^2-h''h-(q-n-1)(h')^2+(2q-2n-3)hh'-(q-n-3)h^2=c_0.
\end{equation}
By (4.12) and (4.13), we get $q-n-3=0$ and $h$ is a constant and $\lambda=0$ and $c_0=0$. So we have
\begin{thm}
Let $M_2=\mathbb{R}^{(1,2)}\times _\mu N^{(q,n)}$ and $g_\mu=g_1+h^2g_2$and $P=\partial_t$. Then
    $(M_2,g_\mu,\nabla^{\mu})$ is Einstein with the Einstein constant
    $\lambda$ if and only if $\lambda=0$, $h=c_1$, $q-n-3=0$ and $( N^{(q,n)},\nabla^{L,N})$ is Einstein with the Einstein constant $0$.
    \end{thm}

\vskip 0.5 true cm

\section{Acknowledgements}

The author was supported in part by  NSFC No.11771070.

\vskip 0.5 true cm


\bigskip

\noindent {\footnotesize {\it Yong Wang} \\
{School of Mathematics and Statistics, Northeast Normal University, Changchun 130024, China}\\
{Email: wangy581@nenu.edu.cn}


\begin{thebibliography}{20}
\bibitem{ARS}
L. Al\'{i}as, A. Romero, M. S\'{a}nchez, Spacelike hypersurfaces of constant mean curvature and Clabi-Bernstein type problems,
Tohoku Math. J. 49(1997) 337-345.

\bibitem{BO}
R. Bishop, B. O'Neill, Manifolds of negative curvature, Trans. Am. Math. Soc. 145(1969) 1-49.

\bibitem{BG}
A. Bruce, J. Grabowski, Riemannian structures on $\mathbb{Z}_2^n$-manifolds, Mathematics, 2020, 8, 1469.

\bibitem{DD}
F. Dobarro, E. Dozo, Scalar curvature and warped products of Riemannian manifolds, Trans. Am. Math. Soc. 303(1987) 161-168.

\bibitem{EJK}
P. Ehrlich, Y. Jung, S. Kim, Constant scalar curvatures on warped product manifolds, Tsukuba J. Math. 20(1996) No.1 239-265.

\bibitem{GDMVR}
F. Gholami, F. Darabi, M. Mohammadi, S. Varsaie, M. Roshande, Einstein equations with cosmological constant in super Space-Time, arXiv:2108.11437.

\bibitem{Go}
O. Goertsches, Riemannian supergeometry. Math. Z. 260 (2008), no. 3, 557-593.

\bibitem{Ha}
H. Hayden, Subspace of a space with torsion, Proc. Lond. Math. Soc. 34(1932) 27-50.

\bibitem{SO}
S. Sular, C. \"{O}zgur, Warped products with a
semi-symmetric metric connection, Taiwanese J. Math. 15(2011) no.4 1701-1719.

\bibitem{W1}
Y. Wang, Multiply warped products with a semisymmetric metric connection, Abstr. Appl. Anal. 2014, Art. ID 742371, 12 pp.

\bibitem{Ya}
K. Yano, On semi-symmetric metric connection,
Rev. Roumaine Math. Pures Appl. 15(1970) 1579-1586.


\end{thebibliography}
\end{document}